\documentclass[11pt,twoside,a4paper,reqno]{amsart}

\usepackage[a4paper,width=160mm,top=25mm,bottom=30mm]{geometry}
\usepackage{amsmath}
\usepackage{amssymb}
\usepackage{amsthm}
\usepackage{mathrsfs}
\usepackage{mathtools}
\mathtoolsset{showonlyrefs}

\usepackage{graphicx}
\usepackage{float}
\usepackage{caption}
\usepackage[utf8]{inputenc}
\usepackage[T1]{fontenc}
\usepackage[english]{babel}

\usepackage{enumerate}
\usepackage[inline]{enumitem}
\setlist[enumerate]{leftmargin=*, label=\arabic*.}
\usepackage{xcolor}
\usepackage{algorithm}
\usepackage{algpseudocode}
\usepackage{todonotes}
\setuptodonotes{inline}
\usepackage{hyperref}
\usepackage{footnote}
\makesavenoteenv{abstract}

\usepackage[
  style=alphabetic,
  doi=false,
  isbn=false,
  url=false,
  eprint=false
]{biblatex}
\theoremstyle{plain}
\newtheorem{theorem}{Theorem}

\newtheorem{proposition}[theorem]{Proposition}

\newtheorem{lemma}{Lemma}

\theoremstyle{remark}

\theoremstyle{definition}

\graphicspath{{./Images/}}
\newtheorem{claim}{Claim}

\title[Counterexamples for site minimal cutsets]{Counterexamples to the site-percolation analogues\\ of the Easo--Severo--Tassion cutset theorems}
\author{Joel Bassil}

\begin{document}

\begin{abstract}
Easo, Severo and Tassion \cite{easoCountingMinimalCutsets2025} proved two theorems about minimal \textit{edge} cutsets in infinite graphs: that non-triviality of the uniform critical percolation parameter is equivalent to exponential growth of the number of minimal edge cutsets of size $n$ separating a vertex from infinity, and that uniform transience is an alternative sufficient assumption. We show that the naive extensions of (the original contribution to) these theorems to site percolation and \textit{vertex} cutsets are false; our counterexample involves modifying a rooted binary tree and replacing edges with gadgets that allow many choices for where to block (in constructing a cutset) one of two local routes, while keeping every vertex uniformly close to a supercritical, uniformly transient shortcut skeleton.
\end{abstract}

\maketitle

\vspace{-2.5em}
\renewcommand{\abstractname}{AI Use Statement}    
\begin{abstract}
The counterexample in this note was found solely by ChatGPT 5.5 pro, which also produced the initial draft of its write-up. The author takes full responsibility for the mathematical content of this note and has independently verified it themselves, as well as citing relevant theory where appropriate. Readers are encouraged to verify this counterexample for themselves.
\end{abstract}

\section{Introduction and statements}

Let $G=(V,E)$ be an infinite, connected, locally finite graph. A finite set $W\subseteq V$ is a \emph{vertex cutset} from $v\in V\setminus W$ to infinity if $v$ lies in a finite connected component of the graph induced by $V\setminus W$. It is \emph{minimal} if no proper subset of $W$ is a vertex cutset from $v$ to infinity. Write
\[
  \mathcal{Q}^{\mathrm{site}}_n(v):=\bigl\{W\subseteq V:\ W \text{ is a minimal vertex cutset from } v \text{ to } \infty,\ |W|=n\bigr\},
\]
\[
  q^{\mathrm{site}}_n:=\sup_{v\in V}\bigl|\mathcal{Q}^{\mathrm{site}}_n(v)\bigr|,
  \qquad
  \kappa^{\mathrm{site}}(G):=\sup_{n\ge 1}\bigl(q^{\mathrm{site}}_n\bigr)^{1/n},
\]
with the conventions $|\emptyset|:=0$ and $\infty^{1/n}:=\infty$.

For Bernoulli site percolation of parameter $p\in[0,1]$ on $G$, in which each vertex is open independently with probability $p$, let $\mathbb{P}_p$ denote the law, let $\{v\leftrightarrow\infty\}$ be the event that $v$ is open and lies in an infinite open connected component, and set
\[
  \theta_v(p):=\mathbb{P}_p(v\leftrightarrow\infty),
  \qquad
  \theta^*(p):=\inf_{v\in V}\theta_v(p),
\]
\[
  p_c^*(G):=\inf\bigl\{p\in[0,1]:\theta^*(p)>0\bigr\}.
\]

Finally, regard $G$ as an electrical network with unit conductance on every edge, let $\mathcal{C}_G(v\leftrightarrow\infty)$ denote the effective conductance from $v$ to infinity, and call $G$ \emph{uniformly transient} if
\begin{equation}\label{eq:ut}
  \inf_{v\in V}\mathcal{C}_G(v\leftrightarrow\infty)>0.
\end{equation}
Equivalently, $\inf_{v\in V}\bigl[d_v\,\mathbb{P}_v(\forall t\ge 1: X_t\neq v)\bigr]>0$, where $(X_t)_{t\ge0}$ is simple random walk on $G$ and $d_v$ is the degree of $v$.

The following two statements are the naive site percolation/vertex cutset analogues of (the original contributions of) \cite{easoCountingMinimalCutsets2025}.

\begin{claim}\label{conj:A}
For every infinite, connected, locally finite graph $G$,
\[
p_c^*(G)<1 \implies \kappa^{\mathrm{site}}(G)<\infty .
\]
\end{claim}

\begin{claim}\label{conj:B}
Every infinite, connected, locally finite, uniformly transient graph $G$ satisfies 
\[
\kappa^{\mathrm{site}}(G)<\infty.
\]
\end{claim}

The purpose of this note is to record that both Claim~\ref{conj:A} and Claim~\ref{conj:B} are false, and refuted by a single graph.

\begin{theorem}\label{thm:main}
There exists an infinite, connected, locally finite graph $G$ with a distinguished vertex $\rho$ such that
\begin{enumerate}
  \item[(i)] $\displaystyle\inf_{z\in V(G)}\mathbb{P}_{p}(z\leftrightarrow\infty)>0$ for every $p>2^{-1/2}$, and in particular $p_c^*(G)<1$;
  \item[(ii)] $\displaystyle\mathcal{C}_G(z\leftrightarrow\infty)\ge \tfrac14$ for every $z\in V(G)$, so $G$ is uniformly transient;
  \item[(iii)] $\displaystyle\sup_{n\ge1}\bigl|\mathcal{Q}^{\mathrm{site}}_n(\rho)\bigr|^{1/n}=\infty$, and hence $\kappa^{\mathrm{site}}(G)=\infty$.
\end{enumerate}
Consequently Claims~\ref{conj:A} and~\ref{conj:B} are both false.
\end{theorem}

The graph is constructed in Section~\ref{sec:constr}; parts (i), (iii) and (ii) are proved in Sections~\ref{sec:perc}, \ref{sec:count} and~\ref{sec:trans} respectively.

\section{The construction}\label{sec:constr}

Let $T$ be the rooted binary tree with root $\rho$, so that $\rho$ is at level $0$ and every vertex has exactly two children. For $m\ge1$ let $E_m$ denote the set of edges of $T$ joining level $m-1$ to level $m$, so
\[
  |E_m|=2^m,
\]
and set
\[
  L_m:=m^2 .
\]

We build a graph $G$ from $T$ by replacing every tree edge $e=uv\in E_m$, where $u$ is the parent and $v$ the child, with the following finite gadget. Delete the edge $uv$ and introduce new vertices
\[
  c_e,\ x_{e,1},\ x_{e,2},\ \dots,\ x_{e,L_m}.
\]
Add the \emph{long path}
\[
  u - x_{e,1} - x_{e,2} - \cdots - x_{e,L_m} - v,
\]
together with the \emph{shortcut} edges $uc_e$ and $c_ev$, and the \emph{spoke} edges $c_ex_{e,j}$ for $1\le j\le L_m$.

We call the vertices of $T$ the \emph{original} vertices of $G$, the vertices $c_e$ the \emph{hubs}, and the vertices $x_{e,j}$ the \emph{path vertices}. We retain the notions of level and of descendant from $T$.

The graph $G$ is infinite and connected. It is locally finite: an original vertex has degree at most $6$ (two edges into the gadget above it and two into each of the two gadgets below it), a path vertex has degree $3$, and
\[
  \deg(c_e)=L_m+2=m^2+2 \qquad (e\in E_m),
\]
which is finite for each fixed $e$. Note that $G$ has unbounded degree.

Throughout we write $H\subseteq G$ for the \emph{shortcut skeleton}: the subgraph consisting of all original vertices and all hubs, together with the edges $uc_e$ and $c_ev$ for every $e=uv$. Thus $H$ is exactly the rooted binary tree $T$ with each edge subdivided once.

\section{Uniform site percolation}\label{sec:perc}

\begin{lemma}\label{lem:perc}
Let $p>2^{-1/2}$ and set
\[
  a:=\frac{2p^2-1}{p^4}\in(0,1].
\]
Then $\inf_{z\in V(G)}\mathbb{P}_p(z\leftrightarrow\infty)\ \ge\ p^3a\ >\ 0$. In particular $p_c^*(G)<1$.
\end{lemma}

\begin{proof}
Since $H$ is a subgraph of $G$, any open connection to infinity inside $H$ is also one inside $G$, so it suffices to work in $H$.

Fix an original vertex $u$ and let $a_N$ denote the probability that, conditionally on $u$ being open, $u$ is connected to level $N$ below it by an open path in the descendant shortcut subtree of $u$. Then $a_0=1$, and decomposing on the two forward branches of $u$ gives
\[
  a_N=f(a_{N-1}),\qquad f(s):=1-\bigl(1-p^2s\bigr)^2 ,
\]
because a given branch, through the hub $c_e$ to the child $v$, succeeds precisely if $c_e$ is open (probability $p$), $v$ is open (probability $p$) and $v$ reaches level $N-1$ below itself (probability $a_{N-1}$), and the two branches are independent. $f$ is increasing on $[0,1]$, so if $f$ has a positive fixed point in $[0,1]$, we may use it as a lower bound for the recursion and thus for $a_N$.

Indeed, the two nonnegative fixed points of $f$ are $s=0$ and
\[
s=a:=\frac{2p^2-1}{p^4},
\]
and $a>0$ exactly when $p^2>1/2$, while $a\le 1$ because $p^4-2p^2+1=(p^2-1)^2\ge0$. Suppose for induction that $a_{N-1} \in [a,1]$. Since $f$ is increasing on $[0,1]$, $a_N = f(a_{N-1})\geq f(a) = a$. Since also $a_0=1\geq a$, by induction, we have that $a_N\geq a$ for every $N$.

The events $\{u\text{ reaches level }N\}$ are nested decreasing in $N$, so by continuity from above,
\[
  \mathbb{P}_p\bigl(u\leftrightarrow\infty \mid u \text{ open}\bigr)\ \ge\ \lim_{N\to\infty}a_N\ \ge\ a .
\]
Hence every original vertex satisfies $\mathbb{P}_p(u\leftrightarrow\infty)\ge pa$.

Let $c_e$ be the hub of the gadget over $e=uv$, with $v$ the child endpoint. On the event that $c_e$ is open, $v$ is open, and $v$ has forward survival in the descendant shortcut subtree, the vertex $c_e$ is connected to infinity. Therefore $\mathbb{P}_p(c_e\leftrightarrow\infty)\ge p^2a$.

Similarly, for a path vertex $x_{e,j}$, the route $x_{e,j}-c_e-v$ shows that the event that $x_{e,j}$, $c_e$ and $v$ are all open and $v$ has forward survival is contained in $\{x_{e,j}\leftrightarrow\infty\}$, so $\mathbb{P}_p(x_{e,j}\leftrightarrow\infty)\ge p^3a$.

These three cases exhaust $V(G)$, and $p^3a\le p^2a\le pa$, which gives the stated bound.
\end{proof}

\section{Super-exponential cutset growth at the root}\label{sec:count}

Fix $m\ge1$. For each edge $e\in E_m$ choose an index $j_e\in\{1,\dots,L_m\}$, and for the resulting function $j:E_m\to\{1,\dots,L_m\}$ define
\[
  W(j):=\{c_e:\ e\in E_m\}\ \cup\ \{x_{e,j_e}:\ e\in E_m\},
  \qquad\text{so}\qquad
  |W(j)|=2|E_m|=2^{m+1}.
\]
Note that $\rho\notin W(j)$.

\begin{lemma}\label{lem:cut}
For every $m\ge1$ and every choice function $j$, the set $W(j)$ is a minimal vertex cutset from $\rho$ to infinity.
\end{lemma}

\begin{proof}
First, $W(j)$ is a vertex cutset. Let $e=uv\in E_m$ with $u$ at level $m-1$. In $G\setminus W(j)$ the shortcut route $u-c_e-v$, or any other route using $c_e$ (for example, those which traverse spoke edges), is destroyed by the removal of $c_e$, and the long route $u-x_{e,1}-\cdots-x_{e,L_m}-v$ is destroyed by the removal of $x_{e,j_e}$; these are the only routes through the gadget of $e$. Hence no level-$m$ gadget can be traversed. Every path from $\rho$ to infinity in $G$ must traverse some level-$m$ gadget, so the connected component of $\rho$ in $G\setminus W(j)$ is contained in the union of the (finitely many) gadgets of levels $1,\dots,m$ together with the original vertices of levels $0,\dots,m-1$. Each such gadget is finite, so this component is finite.

Now let $y\in W(j)$; we show that $W(j)\setminus\{y\}$ is not a vertex cutset. Observe that $W(j)$ meets only gadgets of level $m$, so all gadgets of levels $1,\dots,m-1$ and all gadgets of levels $>m$ are untouched.

If $y=c_e$ for some $e=uv\in E_m$, then in $G\setminus(W(j)\setminus\{y\})$ the vertex $u$ is connected to $\rho$ through untouched gadgets, the route $u-c_e-v$ is available, and below $v$ the entire descendant decorated subtree is untouched, so $\rho$ is connected to infinity.

If $y=x_{e,j_e}$ for some $e=uv\in E_m$, the same argument applies with the long path $u-x_{e,1}-\cdots-x_{e,L_m}-v$ in place of the shortcut.

Finally, if $W'\subsetneq W(j)$ is any proper subset, pick $y\in W(j)\setminus W'$; then $G\setminus W'$ contains $G\setminus(W(j)\setminus\{y\})$ as a subgraph, in which $\rho$ is already connected to infinity. So no proper subset of $W(j)$ is a vertex cutset, and $W(j)$ is minimal.
\end{proof}

\begin{proposition}\label{prop:count}
$\displaystyle\sup_{n\ge1}\bigl|\mathcal{Q}^{\mathrm{site}}_n(\rho)\bigr|^{1/n}=\infty$, and consequently $\kappa^{\mathrm{site}}(G)=\infty$.
\end{proposition}

\begin{proof}
Distinct choice functions $j:E_m\to\{1,\dots,L_m\}$ give distinct sets $W(j)$, since $j$ is recoverable from $W(j)$. By Lemma~\ref{lem:cut} all of them lie in $\mathcal{Q}^{\mathrm{site}}_{n_m}(\rho)$ with $n_m:=2^{m+1}$, so
\[
  \bigl|\mathcal{Q}^{\mathrm{site}}_{n_m}(\rho)\bigr|\ \ge\ L_m^{|E_m|}=\bigl(m^2\bigr)^{2^m}.
\]
Therefore
\[
  \bigl|\mathcal{Q}^{\mathrm{site}}_{n_m}(\rho)\bigr|^{1/n_m}\ \ge\ \Bigl(\bigl(m^{2}\bigr)^{2^m}\Bigr)^{1/2^{m+1}}=m ,
\]
and letting $m\to\infty$ gives the first claim. Since $q^{\mathrm{site}}_{n}\ge|\mathcal{Q}^{\mathrm{site}}_{n}(\rho)|$ for every $n$, we get $\kappa^{\mathrm{site}}(G)=\infty$.
\end{proof}

\section{Uniform transience}\label{sec:trans}

\begin{proposition}\label{prop:trans}
Give every edge of $G$ unit conductance. Then $R_G(z\leftrightarrow\infty)\le4$, equivalently $\mathcal{C}_G(z\leftrightarrow\infty)\ge\frac14$, for every $z\in V(G)$. In particular, $G$ is uniformly transient.
\end{proposition}

\begin{proof}
Fix an original vertex $u$ and work first inside the descendant shortcut subtree of $u$ in $H$, which is a binary tree with each edge subdivided once. For $N\ge0$ let $r_N$ be the effective resistance from $u$ to level $N$ below $u$ in this network, with all level-$N$ vertices wired together. Then $r_0=0$, and decomposing on the two branches of $u$, each consisting of two unit resistors in series (through the hub) followed by a copy of the depth-$(N-1)$ network, gives two parallel branches of resistance $2+r_{N-1}$ each, so by the series and parallel laws \cite[Section~2.3]{lyonsProbabilityTreesNetworks2017},
\[
  r_N=\frac{2+r_{N-1}}{2}=1+\frac{r_{N-1}}{2},
  \qquad\text{whence}\qquad
  r_N=2\bigl(1-2^{-N}\bigr).
\]
Since the effective resistance from a vertex to infinity is the increasing limit of the effective resistances to the boundaries of an exhaustion \cite[Section~2.2]{lyonsProbabilityTreesNetworks2017}, the resistance from $u$ to infinity within the descendant shortcut subtree equals $\lim_{N\rightarrow \infty} r_N=2$. That subtree is a subnetwork of $G$, so by Rayleigh monotonicity \cite[Section~2.4]{lyonsProbabilityTreesNetworks2017}
\[
  R_G(u\leftrightarrow\infty)\le 2 \qquad\text{for every original vertex } u.
\]

Effective resistance is a metric on $V(G)\cup\{\infty\}$; this is a straightforward extension of the finite-network case \cite[Exercise~2.67]{lyonsProbabilityTreesNetworks2017} using the exhaustion definition of effective resistance to infinity. If $c_e$ is the hub of the gadget over $e=uv$ with $v$ the child endpoint, then $c_e\sim v$ gives $R_G(c_e\leftrightarrow v)\le1$, so
\[
  R_G(c_e\leftrightarrow\infty)\le R_G(c_e\leftrightarrow v)+R_G(v\leftrightarrow\infty)\le 1+2=3 .
\]
If $x_{e,j}$ is a path vertex, then $x_{e,j}\sim c_e$ and the same argument gives
\[
  R_G(x_{e,j}\leftrightarrow\infty)\le 1+R_G(c_e\leftrightarrow\infty)\le 4 .
\]
These cases exhaust $V(G)$.
\end{proof}

\begin{proof}[Proof of Theorem~\ref{thm:main}]
Take $G$ and $\rho$ as in Section~\ref{sec:constr}. Part (i) is Lemma~\ref{lem:perc}, part (ii) is Proposition~\ref{prop:trans}, and part (iii) is Proposition~\ref{prop:count}. Part (i) together with part (iii) contradicts Claim~\ref{conj:A}, and part (ii) together with part (iii) contradicts Claim~\ref{conj:B}.
\end{proof}

\printbibliography

@article{easoCountingMinimalCutsets2025,
  title = {Counting Minimal Cutsets and $p_c<1$},
  author = {Easo, Philip and Severo, Franco and Tassion, Vincent},
  date = {2025},
  journaltitle = {Forum of Mathematics, Pi},
  shortjournal = {Forum of Mathematics, Pi},
  volume = {13},
  eprint = {2412.04539},
  eprinttype = {arXiv},
  eprintclass = {math},
  pages = {e23},
  issn = {2050-5086},
  doi = {10.1017/fmp.2025.10011},
  url = {http://arxiv.org/abs/2412.04539},
  urldate = {2026-01-27}
}

@book{lyonsProbabilityTreesNetworks2017,
  title = {Probability on {{Trees}} and {{Networks}}},
  author = {Lyons, Russell and Peres, Yuval},
  date = {2017},
  series = {Cambridge {{Series}} in {{Statistical}} and {{Probabilistic Mathematics}}},
  publisher = {Cambridge University Press},
  location = {Cambridge},
  doi = {10.1017/9781316672815},
  url = {https://www.cambridge.org/core/books/probability-on-trees-and-networks/4249FD4F1D64691AAD5314AEBFAC7ABF},
  urldate = {2026-02-17},
  isbn = {978-1-107-16015-6}
}

\end{document}